\documentclass[10pt,reqno]{amsart}
\usepackage{amsmath, latexsym, amsfonts, amssymb,amsthm, amscd,epsfig,enumerate}
\usepackage{subfigure,color, float}

\advance\hoffset -.75cm

\usepackage[usesnames]{xcolor}

\definecolor{refkey}{gray}{.75}
\definecolor{labelkey}{gray}{.75}

\newcommand{\C}{\mathbb C}

\newcommand{\Z}{\mathbb Z}
\newcommand{\N}{\mathbb N}

\newcommand{\Zd}{{{\mathbb Z}^d}}

\newcommand{\I}{\mathbf 1}

\newcommand{\eps}{\varepsilon}

\newcommand{\pr}{\mathbb P}

\newtheorem{teo}{Theorem}[section]

\newtheorem{cor}[teo]{Corollary}
\newtheorem{rem}[teo]{Remark}

\newtheorem{defn}[teo]{Definition}
\newtheorem{exmp}[teo]{Example}

\title
{Global survival of branching random walks \\ and 
tree-like branching random walks
}

\author[D.~Bertacchi]{Daniela Bertacchi}
\address{D.~Bertacchi, Dipartimento di Matematica e Applicazioni,
Universit\`a di Milano--Bicocca,
via Cozzi 53, 20125 Milano, Italy.}
\email{daniela.bertacchi\@@unimib.it}

\author[C.~F.~Coletti]{Cristian F.~Coletti}
\address{C.~Coletti, Centro de Matem\'atica, Computa\c{c}\~ao e Cogni\c{c}\~ao, 
Universidade Federal do ABC, Avenida dos Estados, 5001- Bangu - Santo Andr\'e - S\~ao Paulo, Brasil}
\email{cristian.coletti@ufabc.edu.br}

\author[F.~Zucca]{Fabio Zucca}
\address{F.~Zucca, Dipartimento di Matematica,
Politecnico di Milano,
Piazza Leonardo da Vinci 32, 20133 Milano, Italy.}
\email{fabio.zucca\@@polimi.it}

\date{}

\begin{document}

\begin{abstract}
The reproduction speed of a continuous-time branching random walk is proportional to a
positive parameter $\lambda$. There is a threshold for $\lambda$, which is called $\lambda_w$, that
separates almost sure global extinction from global survival. Analogously, there exists another threshold
$\lambda_s$ below which any site is visited almost surely a finite number of times (i.e.~local extinction) while above
it there is a positive probability of visiting every site infinitely many times.
The local critical parameter $\lambda_s$ is completely understood
and can be computed as a function of the reproduction rates. On the other hand,
only for some classes of branching random walks it is known that the global critical
parameter $\lambda_w$ is the inverse of a certain function of the reproduction rates, which we denote by $K_w$.
We provide here new sufficient conditions which guarantee that the global
critical parameter equals $1/K_w$.
This result extends previously known results
for branching random walks on multigraphs and general branching random walks.
We show that these sufficient conditions are satisfied by periodic tree-like branching random walks.
We also discuss the critical parameter and the critical behaviour of continuous-time branching processes
in varying environment. 
So far, only examples where $\lambda_w=1/K_w$ were known; here we
provide an example where $\lambda_w>1/K_w$.
\end{abstract}


\maketitle


\section{Introduction}
\label{sec:intro}

\subsection{Background}

The theory of time-homogeneous branching processes dates back to the work of Galton and Watson
(\cite{cf:GW1875}) and the characterization of survival of these processes is very simple:
the expectation of the offspring distribution must be strictly larger than 1.
One way to add complexity is to study the process on a spatial structure: the individuals
live on 
a set $X$ and randomly reproduce; the offspring are dispersed in $X$ according to a probability
distribution.
If we look at the trajectory of lineages, they can be seen as random walks, which branch whenever
an individual has more than one child, whence
the name Branching Random Walk (briefly, BRW) for the process.
In a BRW survival can be global or local, meaning that with positive probability there
will always be someone alive on the graph (global survival) or on a given vertex (local survival).
Clearly local survival is more restrictive than global survival and both situations
become more likely when individuals get more prolific.
In continuous time an easy way to tune reproductions (and to have markovianity) is
to fix $\lambda>0$ and
to attach to each particle living at $x$ and to some of the couples $(x,y)$ an exponential clock
whose parameter is proportional to $\lambda$.
The same is repeated for all particles, sites and edges.
Whenever the clock rings, the corresponding particle at $x$
(if still alive) places an offspring at $y$.
The probability of global and local survival are nondecreasing functions of $\lambda$.

\subsection{The model}

To be precise, let us define the model. We consider $(X,K)$ where $X$ is a countable (or finite) set and $K=(k_{xy})_{x,y \in X}$ is
a matrix of nonnegative entries
such that $\sum_{y \in X} k_{xy} <+\infty$ for all $x \in X$.
The couple $(X,K)$ identifies the BRW, that is a family of continuous-time processes,
depending on a positive parameter $\lambda$.
With a slight abuse of notation, when there is no ambiguity, we omit the dependence on
$\lambda$ and call BRW also the process with a fixed $\lambda$. 
The state space is $\mathbb N^X$. We denote by $\{\eta_t\}_{t\ge0}$ the realization of the process for a fixed $\lambda$: 
$\eta_t(x)$ is the number of particles alive at time $t$, at site $x\in X$.
The evolution in each vertex $y$ is given by the transition rates
\[
 \begin{split}
  i\to i-1 & \text{ at rate }i,\text{ for all }i\ge1,\\
  i\to i+1 & \text{ at rate }\lambda\sum_{x\in X}k_{xy}\eta_{t}(x),\text{ for all }i\ge0.
 \end{split}
\]
Informally, each individual dies at rate 1 and for each couple of sites $(x,y)$ such that $k_{xy}>0$, 
individuals at $x$ give birth to a new individual
at $y$, at rate $\lambda k_{xy}$. All the Poisson clocks are independent. 
We observe (see \cite[Remark~2.1]{cf:BZ14-SLS}) that the assumption
of a non-constant death rate does not represent a
significant generalization. 

Starting with one particle at time 0, it
is well-known that there exist two critical parameters, $\lambda_w\le\lambda_s$
such that for $0\le\lambda<\lambda_w$ the BRW goes extinct almost surely; for
$\lambda\in(\lambda_w,\lambda_s]$ there is local extinction but global survival;
for $\lambda_s<\lambda$ there is local survival (when $\lambda=\lambda_w$, depending
on the cases, there can be global extinction or global survival). The two parameters
$\lambda_s$ and $\lambda_w$ are called local 
and global critical parameters, respectively.
These parameters in principle depend on the starting vertex (in this case they 
are denoted by $\lambda_s(x)$ and $\lambda_w(x)$), but they
are actually equal for all vertices in the same irreducible class
(for more details, also on the critical cases, see Section \ref{sec:basic}).

\subsection{Literature}
We should mention that under the name BRW one can find, in the literature,
several kinds of processes: for instance processes in discrete time, with no death, where parents
randomly walk either before or after breeding, on continuous space, in random environment or with multiple types
(\cite{cf:Bramson92,cf:denHollMensh,cf:HueLal2000,cf:MachadoMenshikovPopov,cf:MP03}
just to name a few).
At least when one wants to characterize survival and extinction, some of these variants can be treated using
similar techniques. In the present paper, we will refer only to continuous-time BRWs as defined above. 

The characterization of $\lambda_s$ in terms of the matrix $K$ has been known
for quite a while. Indeed, Pemantle and Stacey proved (\cite[Lemma 3.1]{cf:PemStac1}) that if
the infinite matrix $K$ is irreducible and $k_{xy} \in \{0,1\}$, then $\lambda_s=1/M$,
where $M=\lim_{n\to\infty}(k^{(2n)}_{xx})^{1/2n}$ and $k^{(2n)}_{xx}$ is the $(x,x)$ element of the $2n$-th power of the matrix 
$K$.
This result has been extended to irreducible BRWs on multigraphs by
\cite[Theorem 3.1]{cf:BZ} and then to generic BRWs 
by \cite[Theorem 4.1]{cf:BZ2}, where $\lambda_s$ may depend on the starting vertex of the process.
The behaviour at $\lambda=\lambda_s$ is also understood: \cite[Theorem 3.5]{cf:BZ}
and \cite[Theorem 4.7]{cf:BZ2} prove that there is almost sure extinction
in continuous time, in the case of multigraphs and in general, respectively.
The discrete-time case has been described in \cite[Theorem 4.1]{cf:Z1}.
The critical behaviour was also investigated independently, with different
techniques, in \cite{cf:Muller08}.

The characterization of $\lambda_w$ is more challenging and is the main aim of
this paper. If $X=\Zd$ and $K$ is the adjacency matrix of the lattice, 
then $\lambda_w=\lambda_s$:
it is also said that there is no weak phase. The absence of weak phase can
be found in many cases which, like $\Zd$, are nonamenable.
Nonamenability by itself is neither necessary nor sufficient for
$\lambda_w=\lambda_s$, as proven in \cite{cf:PemStac1}.
Nevertheless adding some kind of regularity to the graph, like
quasi-transitivity (see \cite[Theorem 3.1]{cf:Stacey03}) or some
more general regularity (see \cite[Theorem 3.6]{cf:BZ}) turns
nonamenability into an equivalent condition for the absence of weak phase.
The presence of the weak phase was first observed on regular trees $\mathbb T_d$
(where $K$ is its adjacency matrix)
and in that case, $\lambda_w=1/d$ was computed in \cite{cf:MadrasSchi}
(note that vertex transitivity makes $\lambda_w$ easy to determine, since
the total progeny is a Galton-Watson process).
When either the graph or $K$ lack regularity, the characterization of $\lambda_w$
is not obvious. For instance, on Galton-Watson trees, only some bounds for
$\lambda_w$ are known (see \cite{cf:PemStac1, cf:Su2014}).
In \cite{cf:BZ2} the following useful characterization of $\lambda_w$ has been proven
 \begin{equation}\label{eq:characterizationlambdaw}
  \lambda_w(x)=\inf\{\lambda \in \mathbb{R} \colon \exists \mathbf{v} \in l^{\infty}(X), \mathbf{v}(x)>0, \lambda K \mathbf{v} \ge \mathbf{v}\}.
 \end{equation}
Unfortunately this characterization is not explicit, therefore we aim at finding other expressions for $\lambda_w(x)$.


The general characterization 
$\lambda_s(x)=1/\limsup_{n\to\infty}(k^{(n)}_{xx})^{1/n}$ (see \cite[Theorem 4.1]{cf:BZ2})
can be intuitively explained by the fact that the expected value
of the cardinality of the set of $n$-th generation descendants living at $x$
is $\lambda^nk^{(n)}_{xx}$. Then, moving to $\lambda_w$ it would be natural to
conjecture that $\lambda_w=1/\limsup_{n\to\infty}(\sum_{y\in X} k^{(n)}_{xy})^{1/n}$.
The first thing to note is that this conjecture has to be modified since
in \cite[Example 2]{cf:BZ2} we have a BRW where $\lambda_w=1/\liminf_{n\to\infty}
(\sum_{y\in X} k^{(n)}_{xy})^{1/n}$. We denote by $K_w(x)$ the last limit and then look for
conditions guaranteeing that $\lambda_w(x)=1/K_w(x)$.

\subsection{Main results and discussion}

The main theorem of this paper, Theorem~\ref{thm:main}, states that for a generic continuous-time BRW, two
uniformity conditions, (U1) and (U2), together are sufficient for $\lambda_w=1/K_w$
(see Section~\ref{sec:main} for the definition of these conditions).
We mention here that an adjacency matrix always satisfies (U2) and,
in the case of multigraphs with this choice of $K$, it
was already known that (U1) was a sufficient condition  for $\lambda_w=1/K_w$
(\cite[Theorem 3.2]{cf:BZ}).
So far, in the general case, it was proven that $\lambda_w=1/K_w$
is also true for BRWs which can be projected onto finite spaces, namely the $\mathcal F$-BRWs (\cite[Proposition 4.5]{cf:BZ2}).
Theorem~\ref{thm:main} extends this result, since $\mathcal F$-BRWs satisfy the two
conditions (U1) and (U2), while there are examples of BRWs satisfying the uniformity
conditions without being $\mathcal F$-BRWs (for instance, a periodic tree-like BRW).
Moreover, a uniformity request, more restrictive than (U1), had
proven to be sufficient for $\lambda_w=1/K_w$ (\cite[Proposition 4.6]{cf:BZ2}).

The novelty of our result is that we extend the characterization of $\lambda_w$ from BRWs on multigraphs and $\mathcal{F}$-BRWs
to a more general class
of BRWs.
The proof, in this case, requires a completely new
and different technique, which heavily relies on multidimensional generating functions
and their fixed points (generating function techniques  have proven to be excellent tools in the identification of the extinction
probabilities of a BRW, see for instance \cite{cf:BZ2,cf:Hautp12,cf:Hautp13}).

Theorem~\ref{thm:main} provides 
a characterization for $\lambda_w$ which is in the same spirit 
of the general one known for $\lambda_s$. 
Examples~\ref{exmp:treeandZ}, \ref{exmp:star} and \ref{exmp:backtotheorigin} show that
the conditions of our main theorem are not necessary.
This leads to the second noticeable contribution of our paper. So far, in the literature, only examples
where $\lambda_w=1/K_w$ were known. It is then natural to ask whether this characterization holds for generic
BRWs or not. The answer is negative even in the case of irreducible BRWs.
Indeed in Example~\ref{exmp:finally}, we construct an irreducible BRW where $\lambda_w >1/K_w$;
moreover, we also show that there are reducible BRWs where
 $\lambda_w(x) >1/K_w(x)$ and $\lambda_w(y)=1/K_w(y)$ for some $x,y \in X$ (even though $\lambda_w(x)=\lambda_w(y)$ for all $x,y \in X$). 
Hence, even though general characterizations for $\lambda_w(x)$ in terms of functional inequalities are known (see 
equation~\eqref{eq:characterizationlambdaw}),
the search for an explicit expression, similar to the one available for $\lambda_s(x)$, in the case of $\lambda_w(x)$ is still open.

\subsection{Outline}
Here is an outline of the paper. In Section~\ref{sec:basic} we formally define the process, its local and global survival
and the associated critical parameters. We also introduce the generating function $G$ of the BRW and recall
Theorem~\ref{th:survival} which links global survival with some properties of $G$.
In Section~\ref{sec:main} we first prove Theorem~\ref{thm:main}~and then Theorem~\ref{th:mapping},
which gives a sufficient condition for
the uniformity condition (U1) to hold.
Section~\ref{sec:examples} is devoted to examples where the equality $\lambda_w=1/K_w$ holds.
The first example is given by periodic tree-like BRWs, which we define in this paper, much in the spirit
of \cite{cf:spakulova}. In particular they are a family of
self-similar BRWs, which can be neither quasi-transitive nor $\mathcal F$-BRWs.
To figure an idea of the self-similarity we require, one can think of BRWs on tree-like graphs.
Tree-like structures arise naturally in the context
of complex networks. In particular, many social and information networks present a
large-scale tree-like structure or a hierarchical structure (see
\cite{cf:Adcock}, \cite{cf:Chen} and references therein).
The global survival
of this family of processes could not be treated with the previously known
techniques. The second example is given by continuous-time branching processes in varying
environment: namely branching processes where individuals breed accordingly to a Poisson process whose
parameter depends on the generation. It suffices to interpret generations as space variables
and BRW techniques apply. We also show that, even for such a particular law of the process,
when $\lambda=\lambda_w$ still global extinction and global survival are both possible.
Examples~\ref{exmp:treeandZ} and \ref{exmp:star} show that (U1) is not necessary
for $\lambda_w=1/K_w$ (in the first case $\lambda_w<\lambda_s$, in the second case $\lambda_w=\lambda_s$).
Example~\ref{exmp:backtotheorigin} shows that (U2) is not necessary
for $\lambda_w=1/K_w$. In Example~\ref{exmp:finally} we construct an irreducible BRW where $\lambda_w>1/K_w$. It is
worth mentioning that, in the reducible case, it can even happen that
$\lambda_w=1/K_w$ if the process starts from certain vertices and
$\lambda_w>1/K_w$ if it starts from other vertices (see Example~\ref{exmp:finally}).
Section~\ref{subsec:randomgraphs} is devoted to a brief discussion on the case of random graphs.


\section{Basic definitions} \label{sec:basic}

In this section we recall the main tools and definitions which are needed in the sequel.
In Subsection~\ref{subsec:BRW-surv} we give the definition of reducible/irreducible BRW and
a name to the extinction probabilities. We define formally the critical parameters, noting that they depend
on the irreducible class of the starting vertex. In Subsection~\ref{subsec:geompar} we introduce
the parameter $K_s(x,x)$, which is the reciprocal of $\lambda_s(x)$, and $K_w(x)$, which is
our candidate for the reciprocal of $\lambda_w(x)$. We also discuss the dependence of these parameters
on the site $x$.
Subsection~\ref{subsec:genfun} presents the generating function $G$ of the BRW.
This generating function is the analog of the one in the case of the Galton-Watson process.
The extinction probabilities are fixed points of $G$. Theorem~\ref{th:survival} gives characterizations
of global survival in terms of properties of this generating function.
In Subsection~\ref{subsec:projec} we recall the definition of projection of the BRW $(X,K)$ onto a
BRW $(Y,\widetilde K)$. Projections are useful since the two BRWs share the same $\lambda_w$
and in some cases it is easier to compute the critical parameter in $(Y,\widetilde K)$.
In particular, if $Y$ is finite (that is, $(X,K)$ is a $\mathcal F$-BRW), then global and local survival are equivalent in $(Y,\widetilde K)$
and $\lambda_w$ can be determined by computing $\lambda_s$ (for which the explicit formula is known). 

\subsection{Extinction probabilities and critical parameters}
\label{subsec:BRW-surv}

Given a BRW $(X,K)$, the probability of survival depends on the initial configuration $\eta_0\in \N^X$.
We denote by $\{\eta_t\}_{t \ge 0}$ the process (with a fixed $\lambda$) and
we consider $\eta_0=\delta_x$ (where $x\in X$). In this case, we say that the process \textit{starts at $x$}.

Depending on which elements of the matrix $K$ are strictly positive, the BRW starting at $x$ may or may not be able to reach 
a fixed vertex $y\in X$. To understand which sites are reachable, we associate to the process
a graph $(X,E_K)$ where $(x,y) \in E_K$  if and only if $k_{xy}>0$.
We say that there is a path from $x$ to $y$, and we write $x \to y$, if it is
possible to find a finite sequence $\{x_i\}_{i=0}^n$, $n \in \N$,
such that $x_0=x$, $x_n=y$ and $(x_i,x_{i+1}) \in E_K$,
for all $i=0, \ldots, n-1$. If $x \to y$ and $y \to x$ we write $x \rightleftharpoons y$.
By definition there is always a path of length $0$ from $x$ to $x$. The equivalence 
class $[x]$ 
with respect to $\rightleftharpoons$ is called \textit{irreducible class of $x$}.
We say that the matrix $K$
 and the BRW $(X,K)$ are \textit{irreducible} 
if and only if the graph $(X,E_K)$ is \textit{connected},
otherwise we call it \textit{reducible}.
The irreducibility of $K$ means that, in the BRW, the progeny of any particle can 
spread to any site of the graph.

Let $\pr^x$ be the law of the process which starts at $x$. We define the local and global survival events and
the associated extinction probabilities. Note that, in the literature, local and global survival are sometimes called 
by \textit{strong} and \textit{weak} survival.
\begin{defn}\label{def:survival} $\ $
\begin{enumerate}
 \item 
The process \textsl{survives locally} in $A \subseteq X$, starting from $x \in X$, if\\
$
{\mathbf{q}}(x,A):=1-\pr^x(\limsup_{t \to \infty} \sum_{y \in A} \eta_t(y)>0)<1.
$
\item
The process \textsl{survives globally}, starting from $x$, if
$
\bar {\mathbf{q}}(x):={\mathbf{q}}(x,X) 
<1.
$
\end{enumerate}
\end{defn}
\noindent
From now on, ${\mathbf{q}}(x,y)$ will be a
shorthand for ${\mathbf{q}}(x,\{y\})$. 
Often we will simply say that local survival occurs ``starting from $x$'' or ``at $x$'':
in this case we mean that $\mathbf{q}(x,x)<1$. When there is no survival, we say that there is extinction
and the fact that extinction occurs with probability 1 will be tacitly understood.

Depending on $x \in X$, two critical parameters are associated to the
continuous-time BRW: the \textit{global} 
\textit{survival critical parameter} $\lambda_w(x)$ and the  \textit{local} 
 \textit{survival critical parameter} $\lambda_s(x)$.
They are defined as
\[ 
\begin{split}
\lambda_w(x)&\equiv \lambda_w(x;X,K):=\inf \Big \{\lambda>0\colon \,
\pr^x\Big (\sum_{w \in X} \eta_t(w)>0, \forall t\Big) >0 \Big \},\\
\lambda_s(x)&\equiv \lambda_s(x;X,K):=
\inf\{\lambda>0\colon \,
\pr^x \big(\limsup_{t \to \infty} \eta_t(x)>0 \big) >0
\}.
  \end{split}
\] 
By definition, we have:
for $\lambda < \lambda_w(x)$ almost sure global extinction;
for $\lambda\in (\lambda_w(x),\lambda_s(x)]$ global survival
with positive probability and almost sure local extinction;
for $\lambda> \lambda_s(x)$ global and local survival with positive
probability.
When $\lambda=\lambda_w(x)$ there might be global survival (as in \cite[Example 3]{cf:BZ2}) 
or global extinction (as in the case of
$\mathcal{F}$-BRWs, see Section~\ref{subsec:projec} for details).
The critical parameters depend only on $[x]$: in the irreducible
case we will write $\lambda_s$ and $\lambda_w$ instead of
 $\lambda_s(x)$ and $\lambda_w(x)$ respectively.

 \subsection{Geometrical parameters}
 \label{subsec:geompar}
We define recursively
$k^{(n)}_{xy}:=\sum_{w\in X} k^{(n-1)}_{x w} k_{w y}$
(where $k^{(0)}_{xy}:=\delta_{xy}$); moreover we set
$T^n_x:=\sum_{y \in X} k^{(n)}_{xy}$ and
$\phi^{(n)}_{xy}:=\sum_{x_1,\ldots,x_{n-1} \in X \setminus\{y\}} k_{x x_1} k_{x_1 x_2} \cdots k_{x_{n-1} y}$;
by definition
$\phi^0_{xy}:=0$ for all $x,y \in X$.
Clearly $\lambda^n k^{(n)}_{xy}$ is the average size of the $n$th generation at $y$ of the progeny 
of a particle living at $x$;
$\lambda^n T^n_x$ is the average size of the $n$th generation of the whole progeny 
of a particle living at $x$.
Finally, 
$\lambda^n \phi^{(n)}_{xy}$ is the analog of $k^{(n)}_{xy}$ concerning only paths reaching
$y$ for the first time at the $n$-th step.


We introduce the following
geometrical parameters
\[
K_s(x,y) \equiv K_s(x,y;X,K):=\limsup_{n} (k^{(n)}_{xy})^{1/n}, \qquad
K_w(x)\equiv K_w(x;X,K):=\liminf_{n} (T_x^n)^{1/n}.
\]
In the rest of the paper, whenever there is no ambiguity, we will omit
the dependence on $X$ and $K$.
We recall that $\lambda_s(x)=1/K_s(x,x)$ (\cite[Theorem 3]{cf:BZ2}).
Supermultiplicative arguments imply that
$K_s(x,x) =\lim_{n\to\infty} (k^{(dn)}_{xx})^{1/dn}$ for some $d \in \N$, whence
we have that $K_s(x,x) \le K_w(x)$, for all $x \in X$.

$K_w(x)$ and $K_s(x,y)$ depend only on the irreducible classes $[x]$ and $[y]$.
For an irreducible BRW, we write 
$K_w:=K_w(x,y)$ and $K_s:=K_s(x)$ for all $x,y \in X$.
 Not only $\lambda_s(x)$ and $K_s(x)$ are constant inside each irreducible class, but they
 also depend only on the restriction of the BRW to the irreducible class $[x]$ (that is, they are the same if computed
 for the original BRW or for its restriction to $[x]$).
 This is due to the fact that local survival takes into account paths starting from $x$ and going back to $x$.
 That might not be true for $\lambda_w(x)$ and $K_w(x)$, since when we restrict the BRW to $[x]$ we
 might lose paths from $x$ which exit $[x]$ (in general $\lambda_w(x)$ of the restricted BRW is not smaller
 than the corresponding parameter for the original BRW and the reversed inequality holds for $K_w(x)$).
\begin{rem}\label{rem:lambdaw}
In general, nothing can be said about the relationship between $\lambda_s(x)$ and $\lambda_s(y)$
 for $[x]\neq[y]$. On the contrary, considering global survival, if $x \to y$ then $\lambda_w(x) \le \lambda_w(y)$ (and $K_w(x) \ge K_w(y)$).
 One may wonder under which conditions this inequality may be reversed.
 Given $A\subseteq X$, if we know that the restriction of the BRW to $X\setminus A$ dies out for all $\lambda < \inf\{\lambda_w(y) \colon
 y \in A\}$, then $\lambda_w(x) \ge \inf\{\lambda_w(y) \colon
 y \in A\}$ for all $x \in X$. The arguments are similar to those used in the comparison between a BRW and the associated no-death BRW as in
 \cite[before Proposition 2.1]{cf:BZ14-SLS}
 or \cite[Section 3.2]{cf:BZ4}. Applications can be found in Section~\ref{subsec:BPVE}.
\end{rem}


The following power series can be useful to identify the critical parameters
\[
\begin{split}
H(x,y|\lambda)&:=\sum_{n =0}^\infty k^{(n)}_{xy} \lambda^n,
\qquad \Theta(x|\lambda):=\sum_{n =0}^\infty T_x^n \lambda^n,
\qquad \Phi(x,y|\lambda):=\sum_{n =1}^\infty \phi_{xy}^{(n)} \lambda^n.
\end{split}
\]
Clearly $1/K_s(x,y)$ is the convergence radius of $H(x,y|\lambda)$ and
for all $\lambda \in \C$ such that $|\lambda|<1/\limsup_{n} (T_x^n)^{1/n}$ we have
$\Theta(x|\lambda)=\sum_{y \in Y} H(x,y|\lambda)$. 
The following relations hold (provided that $\lambda$ is such that the involved series converge):
\begin{equation}\label{eq:HTheta}
\begin{split}
H(x,y|\lambda)&=\delta_{xy}+\lambda \sum_{w \in X} k_{xw}
H(w,y|\lambda)\\
& =
\delta_{xy}+\lambda \sum_{w \in X} H(x,w|\lambda)k_{wy}\\
& =\delta_{xy} +\Phi(x,y|\lambda)H(y,y|\lambda),
\\
\Theta(x|\lambda)&=1+\lambda \sum_{w\in X} k_{xw} \Theta(w|\lambda),
\\
\qquad \Phi(x,x|\lambda)&= \lambda \sum_{y \in X, y \not = x} k_{xy} \Phi(y,x|\lambda)+
\lambda k_{xx}.
\end{split}
\end{equation}
Moreover
if $x,y,w \in X$ are distinct vertices such that every path from $x$ to $y$ contains $w$ then
$\Phi(x,y|\lambda)=\Phi(x,w|\lambda)\Phi(w,y|\lambda)$.
We note that, since 
\begin{equation}
H(x,x|\lambda)=\frac{1}{1-\Phi(x,x|\lambda)},
\qquad \forall \lambda \in \C: |\lambda|< \lambda_s(x), 
\end{equation}
we have that $\lambda_s(x)=\max\{ \lambda \geq 0 :\Phi(x,x|\lambda)\leq 1\}$
for all $x \in X$ (remember that $\Phi(x,x|\cdot)$ is left-continuous on $[0,\lambda_s(x)]$
and that $1/(1-\Phi(x,x|\lambda))$ has no analytic prolongation in $\lambda_s(x)$).

%
%

\subsection{The generating function of the BRW}
\label{subsec:genfun}

To each continuous-time BRW one can associate its \textit{discrete-time counterpart}, that is,
a discrete-time BRW which survives/dies if and only if the original BRW
does (see for instance \cite[Section 2.2]{cf:Z1} or 
\cite[Section 2.2 and Remark 2.1]{cf:BZ14-SLS}). 
In this sense the class of continuous-time BRWs can be considered as a subclass of discrete-time BRWs
and we can study its generating function. 
More precisely, let us denote by $\mu_x(f)$ the probability that a particle living at $x$ places exactly $f(y)$
offsprings at site $y$, before its death.
The generating function $G:[0,1]^X \to [0,1]^X$ has $x$ coordinate given by
\begin{equation}
G(\mathbf{z}|x):= \sum_{f \in \Psi} \mu_x(f) \prod_{y \in X} z(y)^{f(y)},
\end{equation}
where $\Psi$ is the space of finitely supported functions in $\N^X$.
This generating function 
has been introduced in \cite[Section 3]{cf:BZ2}
(see also \cite{cf:BZ4, cf:BZ16, cf:Z1} for additional properties).
In the case of the discrete-time counterpart of a continuous-time BRW, 
the $x$-coordinate of $G(\mathbf{z})$ can be written as
\[
 G(\mathbf{z}|x):=\frac{1}{1+\lambda K(\mathbf{1}-{\mathbf{z}})(x)},
 \]
where $\mathbf{1}(x)=1$ 
$K{\mathbf{z}}(x)=\sum_{y \in X} k_{xy}{\mathbf{z}}(y)$ for all $\mathbf{z} \in [0,1]^X$ and $x \in X$.
Note that $G$ is continuous with respect to the \textit{pointwise convergence topology} of $[0,1]^X$  and nondecreasing
with respect to the usual partial order of $[0,1]^X$ (see \cite[Sections 2 and 3]{cf:BZ2} for further details).
As for this partial order, if we say that an element of $[0,1]^X$ is the smallest (respectively largest) among a set of points 
$\mathcal{A}$, we also imply that it is comparable with every element of the specific set $\mathcal{A}$. 
We stress that $\mathbf{z} < \mathbf{w}$ means
$\mathbf{z}(x) \le \mathbf{w}(x)$ for all $x \in X$ and $\mathbf{z}(x_0) < \mathbf{w}(x_0)$ for some $x_0 \in X$.
Moreover, $G$ represents the 1-step reproductions; we denote by $G^{(n)}$ the generating function
associated to the $n$-step reproductions, which is inductively defined as $G^{(n+1)}({\mathbf{z}})=G^{(n)}(G({\mathbf{z}}))$
($G^{(0)}$ is the identity).
The extinction probabilities are fixed points
of $G$ and the smallest fixed point is $\bar {\mathbf{q}}=\lim_{n\to\infty}G^{(n)}(\mathbf{0})$.
Note that if $G(\mathbf{z}) \le \mathbf{z}$, then $\mathbf z \ge \bar {\mathbf{q}}$.

Global survival can be characterized using $G$, as the following theorem claims
(for the proof see \cite[Theorem 4.1]{cf:Z1} and \cite[Theorem 3.1]{cf:BZ14-SLS},
or \cite[Theorem 2.2]{cf:BRZ16}).

\begin{teo}\label{th:survival}
 Consider a BRW and a fixed $x \in X$. The following assertions are equivalent:
 \begin{enumerate}
  \item $\bar {\mathbf{q}}(x)<1$ (i.e.~there is global survival starting from $x$);
  \item there exists
${\mathbf{q}}\in [0,1]^X$ such that ${\mathbf{q}}(x)<1$ and $G(\mathbf{q}) \le \mathbf{q}$;
  \item there exists
${\mathbf{q}}\in [0,1]^X$ such that ${\mathbf{q}}(x)<1$ and $G(\mathbf{q}) = \mathbf{q}$.
 \end{enumerate}
 If $\mathbf{q}$ satisfies either $(2)$ or $(3)$, then $\mathbf{q} \ge \bar {\mathbf{q}}$.
 Moreover, global survival starting from $x$ implies that $\liminf_{n \to \infty} \lambda^n \sum_{y \in X} k_{xy}^{(n)}>0$
 (or, equivalently,$\inf_{n \to \infty} \lambda^n \sum_{y \in X} k_{xy}^{(n)}>0$). 
 \end{teo}
 As a consequence of this theorem we have that $\lambda_w(x) \ge 1/K_w(x)$; indeed if $\lambda < 1/K_w(x)$, then
 $\liminf_{n \to \infty} \lambda \sqrt[n]{\sum_{y \in X} k_{xy}^{(n)}}<1$, hence $\liminf_{n \to \infty} \lambda^n \sum_{y \in X} k_{xy}^{(n)}=0$.
 This immediately implies that if $K_w(x)=0$ then there is extinction for every $\lambda>0$, whence $\lambda_w(x)=+\infty$.
 We observe that
\begin{equation}\label{eq:comparison}
 1/K_w(x) \le \lambda_w(x) \le \lambda_s(x)=1/K_s(x,x), \quad \forall x \in X,
\end{equation}
thus if $K_s(x,x)=K_w(x)$ then $\lambda_s(x) = \lambda_w(x) = 1/K_w(x)$
(see Theorem~\ref{apgs} for an application).

\subsection{Projections}
\label{subsec:projec}
The definition of \textit{projection of a BRW} 
first appeared in \cite{cf:BZ} for multigraphs, in \cite{cf:BZ2} for continuous-time BRWs and
\cite{cf:Z1} for generic discrete-time BRWs (in these papers it was called \textit{local isomorphism}).

\begin{defn}\label{def:projection}
A projection of a BRW $(X,K)$ onto $(Y,\widetilde K)$ is a surjective map $g:X \to Y$, such that
$\sum_{z \in g^{-1}(y)} k_{xz}=\widetilde k_{g(x)y}$
for all $x \in X$ and $y \in Y$.
If there exists a projection of $(X,K)$ onto a finite $(Y,\widetilde K)$, then $(X,K)$ is called $\mathcal{F}$-BRW.
\end{defn}
The main idea is to label the points in $X$ by using the alphabet $Y$.
Particles at $x$ generate children in the set
of vertices with ``label'' $y$, at a total rate which depends only on 
$y$ and $g(x)$.
If $\{\eta_t\}_{t \ge 0}$ is a realization of the BRW $(X, K)$, then $\{\sum_{z\in g^{-1}(\cdot)}\eta_t(z)\}_{t \ge 0}$ is a realization
of the BRW $(Y, \widetilde K)$. 

In particular, 
there is global survival for $(X,K)$, starting from $x$, if and only if
there is global survival for $(Y, \widetilde K)$, starting from $g(x)$. 
This implies that
$\lambda_w(x;X,K)=\lambda_w(g(x);Y,\widetilde K)$, for all $x \in X$ (see for instance \cite[proof of Proposition 4.5]{cf:BZ2} 
or \cite[before Theorem 4.3]{cf:Z1}).
On the other hand, $K_w(x;X,K)=K_w(g(x);Y,\widetilde K)$, for all $x \in X$. Indeed it is easy to prove, by induction on $n$, that 
$\sum_{z \in X} k^{(n)}_{xz}=\sum_{y \in Y} \widetilde k^{(n)}_{g(x)y}$, for all $n \in \N$, $x \in X$.
This also implies that
$\lim_{n \to \infty} \sqrt[n]{\sum_{z \in X} k^{(n)}_{xz}}$ exists if and only if $\lim_{n \to \infty} 
\sqrt[n]{\sum_{y \in Y} \widetilde k^{(n)}_{g(x)y}}$ does.  
It is worth mentioning that, when $(X,K)$ is an $\mathcal{F}$-BRW, then
$\lambda_w(x)=1/K_w(x)$ and there is almost sure global extinction 
for $\lambda=\lambda_w(x)$ (see for instance \cite{cf:BZ, cf:BZ2, cf:BZ4, cf:Z1}).
%

We observe that $(x,y) \in E_K$ implies $(g(x), g(y)) \in E_{\widetilde K}$
 but the converse is not true. In particular,
if $(X,K)$ is irreducible, then $(Y, \widetilde K)$ is irreducible as well, but the converse is not true in general.
If $(X,K)$ is projected onto $(Y, \widetilde K)$ then, for all ${\mathbf{q}} \in [0,1]^Y$ and $x \in X$,
\begin{equation} 
G_X({\mathbf{q}} \circ g|x)=G_Y({\mathbf{q}}|g(x)),
\end{equation}
that is
\[
 \frac{1}{1+\lambda K(\mathbf{1}_X-{\mathbf{q}\circ g})(x)}
 =
 \frac{1}{1+\lambda \widetilde K(\mathbf{1}_Y-{\mathbf{q}})(g(x))},
\]
where $\mathbf{1}_X(x)=\mathbf{1}_Y(y):=1$ for all $x \in X$ and $y \in Y$.
Moreover the following relation between the probabilities of extinctions hold: $\bar {\mathbf{q}}_X= \bar {\mathbf{q}}_Y \circ g$.

\section{Main results}\label{sec:main}

We know that $\lambda_w=1/K_w$
holds in many cases, like finite BRWs or $\mathcal F$-BRWs (but not in general, according to Example~\ref{exmp:finally}).
It is natural to define two uniformity conditions: (U1) is related to how fast the expected number of descendants at generation $n$
gets close to $(\lambda K_w)^n$; (U2) puts an infimum on the positive reproduction rates.
\begin{defn}\label{def:nxcontinuous}
Given a BRW $(X,K)$ and given $\varepsilon>0$, $x\in X$, 
we define $N_{x,\eps}:=\{n\in\N\colon \sum_{y\in X}k_{xy}^{(n)} \ge \left(K_w(x)-\varepsilon\right)^n\}$
and $n_x(\varepsilon):=\min N_{x,\eps}$.
We say that 
\begin{enumerate}[(1)]
\item condition (U1) is satisfied if for all $\eps>0$, $\sup_{x\in X}n_x(\eps)<+\infty$;
 \item condition (U2) is satisfied
 if $\inf\{k_{xy}\colon x,y \text{ such that }k_{xy}>0\}>0$.
 \end{enumerate}
 \end{defn}
Our main result, Theorem~\ref{thm:main}, claims that for irreducible BRWs, if (U1) and (U2) hold, then $\lambda_w(x)=1/K_w(x)$.
Let us give here an informal sketch of the ideas that we employ here. Start the process with one particle at $x$.
By definition of $K_w$, there exists $n_x\in\mathbb N$, such that the expected number of descendants in generation $n_x$ is
close to $(\lambda K_w(x))^{n_x}$. In particular this expected number is larger than 1 for $\lambda>1/K_w$.
We consider a discrete-time BRW, $(X,\widehat K)$, whose offspring are, at generation 1, the 
descendants that $(X,K)$ has in generation $n_x$. We prove survival of $(X,\widehat K)$ by showing
that for its generating function $\overline G$ there exists $s>0$ such that $\overline G(s\mathbf 1_X|x)\le s$ for all $x\in X$.

Corollary~\ref{cor:main} weakens the conditions, since it proves that, in order to have $\lambda_w=1/K_w$,
it suffices that $(X,K)$ can be projected onto
a BRW which satisfies the assumptions of Theorem~\ref{thm:main} (an example of a BRW where the corollary
applies is described after its proof).
We note that it can be tricky to check whether (U1) is satisfied.
Theorem~\ref{th:mapping} provides a sufficient geometrical condition:
the existence of a subset of vertices $Y$ which are not too far from any other vertex in $X$,
and of a family of maps which ``preserve'' the reproduction rates and send a fixed vertex in $X$ onto any chosen vertex in $Y$. 
An  application of this theorem, 
is given in Section~\ref{subsec:treelike} 
where we study the global survival critical parameter of periodic tree-like BRWs.

\begin{teo}\label{thm:main}
Let $(X,K)$ be an irreducible, continuous-time BRW such that $\sup_{x \in X} \sum_{y \in X} k_{xy} <+\infty$.
 If conditions (U1) and (U2) hold  then $\lambda_w=1/K_w$.
\end{teo}

\begin{proof}
By irreducibility $\lambda_w$ and $K_w$ do not depend on $x \in X$.
 Fix $\lambda > 1/K_w$: we want to prove that the 
 $\lambda$-BRW survives.
Choose $\eps > 0$ such that $\lambda \left(K_w - \eps\right) > 1 + \eps$ and let 
$n_x:=n_x(\eps)$ for all $x\in X$. We study 
a discrete-time BRW $(X,\widehat K)$  where $\hat k_{xy}=k^{(n_x)}_{xy}$, for all $x,y\in X$.
This means that, in $(X,\widehat K)$, the 1-step children of a particle living at $x$ are the $n_x$-th generation 
descendants of the particle, in $(X,K)$.
Clearly if $(X,\widehat K)$ survives, so does $(X,K)$. The generating function of $(X,\widehat K)$ is given by
\begin{equation*}
 \overline{G}(\mathbf{z}|x)=G^{(n_x)}(\mathbf{z}|x) 
\end{equation*}
where $G$ is the generating function of $(X,K)$.

Let $\nu_x$ be the distribution of the total number of children of a particle at $x$ in $(X,\widehat K)$. Denote by 
$\widehat G_x$ the 1-dimensional generating function of $\nu_x$ which is given by $\widehat G_x(t)\equiv\overline{G}\left(t\I_X|x\right)$.
Then the mean number of descendants of a particle at $x$ in $(X,\widehat K)$ is
\begin{equation*}
\widehat G_x^{\prime}(1) = \sum_{n=0}^\infty n \nu_x(n) =\lambda^{n_x} T_x^{n_x} > \left(\lambda (K_w -\eps)\right)^{n_x} \geq 1+\eps.
\end{equation*}
Since $(X,K)$ is a continuous-time BRW, then $G({\bf{z}}|x) = 1/(1+\lambda  \sum_{y\in X} k_{xy}(1-{\bf{z}}(y)))$. 
We can determine the first and second moments of the number of $n$-th generation descendants of a particle at $x$, by means of $G$.
Indeed let us denote these moments by $m_{n,x}$ and $m_{n,x}^{(2)}$ respectively.
If ${\bf{z}}=t\bf{\I_X}$ we have
\begin{equation*}
m_{n,x}=\frac{d}{dt}G^{(n)}(t{\bf{\I_X}}|x)\Bigr|_{\substack{t=1}} = \lambda  \sum_{y\in X} k_{xy} \frac{d}{dt}G^{(n-1)}
(t{\bf{\I_X}}|y)\Bigr|_{\substack{t=1}}
= \lambda  \sum_{y\in X} k_{xy} m_{n-1,y}
\end{equation*}
and
\[
\begin{split}
m_{n,x}^{(2)} - m_{n,x} &= \frac{d^2}{dt^2}G^{(n)}(t{\bf{\I_X}}|x)\Bigr|_{\substack{t=1}} \\
&= \sum_{y\in X} \lambda k_{xy} \frac{d^2}{dt^2}G^{(n-1)}(t{\bf{\I_X}}|y)\Bigr|_{\substack{t=1}} + 
2 \left( \sum_{y\in X} \lambda k_{xy} \frac{d}{dt}G^{(n-1)}(t{\bf{\I_X}}|y)\Bigr|_{\substack{t=1}}\right)^2.
\end{split}
\]
We denote by $\xi_{n,x}:=(m_{n,x}^{(2)} - m_{n,x})/m_{n,x}^2$.
Then for all $x\in X$,
\begin{equation}\label{eq:xi}
\xi_{n,x}:=  2 + \frac{ \sum_{y\in X}  \lambda k_{xy} 
\frac{d^2}{dt^2}G^{(n - 1)}(t{\bf{\I_X}}|y)\Bigr|_{\substack{t=1}}}{\left( \sum_{y\in X}  \lambda k_{xy} 
\frac{d}{dt}G^{(n - 1)}(t{\bf{\I_X}}|y)\Bigr|_{\substack{t=1}}\right)^2}
=2 + \frac{ \sum_{y\in X} \lambda k_{xy} \xi_{n-1,y} m^2_{n-1,y}}{\Big ( \sum_{y\in X} \lambda k_{xy} m_{n-1,y}\Big )^2}.
\end{equation}
Define $\xi_n:= \sup_{x \in X} \xi_{n,x}$;
a straightforward computation shows that
$\xi_{1,x}=2=\xi_1$ for all $x \in X$.
From equation~\eqref{eq:xi} we have
\[
 \xi_{n}\le 2+\xi_{n-1}\sup_{x \in X} \frac{ \sum_{y\in X} \lambda k_{xy} m^2_{n-1,y}}{\Big ( \sum_{y\in X} \lambda k_{xy} m_{n-1,y}\Big )^2}
 \le 2 + \xi_{n-1} \sup_{x \in X} \frac{ \sum_{y\in X} \lambda k_{xy} m^2_{n-1,y}}{\Big ( \sum_{y\in X} \sqrt{\lambda k_{xy} \delta}\,  m_{n-1,y}\Big )^2}
 \le 2 + \frac{\xi_{n-1}}{\delta}
\]
where $\delta := \lambda  \inf\{k_{xy}\colon x,y \text{ such that } k_{xy} > 0\}$.
Hence by induction
\begin{equation*}
\xi_{n} \leq 2 \sum_{k=0}^{n-1} \left(\frac{1}{\delta}\right)^k,
\end{equation*}
which implies $\xi_{n_x} \leq M:=2\sum_{k=0}^{N-1} \left(\frac{1}{\delta}\right)^k$ where $N:= \sup_{x \in X} n_x<+\infty$ by condition (U1).

By Theorem~\ref{th:survival} in order to prove survival of $(X,\widehat K)$ it is enough to prove that $\widehat G_x(1-t) \le 1-t$, 
for some $t \in (0,1)$ and for all $x\in X$. Writing the Taylor expansion of $\widehat G_x$ at $1$ and using the monotonicity 
of $\widehat G_x^{\prime\prime}(\cdot)$, we have
\begin{equation*}
 \begin{split}
  \widehat G_x(1-t) 
&\leq 1-m_{n_x,x} t+\frac{t^2}{2}\widehat G_x^{\prime\prime}(1) \\
&=1-m_{n_x,x} t+\frac{t^2}{2}\left(m_{n_x,x}^{(2)}-m_{n_x,x}\right). 
 \end{split}
\end{equation*}

Therefore, $\widehat G_x(1-t) \leq 1-t$ for all $x\in X$ if 
\begin{equation*}
t\le 2 \left(\sup_{x\in X}\frac{m_{n_{x},x}^{(2)}-m_{n_{x},x}}{m_{n_{x},x} - 1}\right)^{-1}.
\end{equation*}
Since $m_{n_x,x} > 1 + \eps$  for any $x \in X$ and
\begin{equation*}
\frac{m_{n_{x},x}^{(2)}-m_{n_{x},x}}{m_{n_{x},x} - 1} = \xi_{n_x,x} \frac{m_{n_{x},x}^2}{m_{n_{x},x} - 1} 
\end{equation*}
we get
\begin{equation*}
\frac{m_{n_{x},x}^{(2)}-m_{n_{x},x}}{m_{n_{x},x} - 1} \leq M \frac{\left(\lambda M^\prime \right)^{2N}}{\eps}
\end{equation*}
where $M^\prime:=\sup_{x \in X} \sum_{y \in X} k_{xy}$. Thus 
the constant solution is obtained by choosing a strictly positive $t \le 2\eps/(M(\lambda M^\prime )^{2N})$.
%
%
\end{proof}
In the previous theorem irreducibility is not necessary:
it suffices that $(X,K)$ is such that $K_w(x)$ does not depend on $x$.
In particular this is the case when $(X,\widetilde K)$ can be projected onto an irreducible BRW, which
leads to the following corollary.
\begin{cor}\label{cor:main}
Let $(X,K)$ be a BRW which can be projected onto an irreducible BRW $(Y,\widetilde K)$ which satisfies the hypotheses
of Theorem~\ref{thm:main}.
Then $\lambda_w(x;X,K)$ and $K_w(x;X,K)$ do not depend on $x \in X$ and $\lambda_w(X,K)=1/K_w(X,K)$.
\end{cor}

\begin{proof}
 The independence comes from the equalities  $\lambda_w(x;X,K)=\lambda_w(g(x);Y, \widetilde K)$ and $K_w(x;X, K)=K_w(g(x);Y, \widetilde K)$ 
and the fact that the two parameters computed
 on $Y$ do not depend on the vertex by irreducibility. It is enough now to apply Theorem~\ref{thm:main} to $(Y, \widetilde K)$. 
\end{proof}
As we remarked before, any BRW with $X$ finite satisfies (U1) and (U2), thus the previous corollary is a generalization, in the
irreducible case, of the results for $\mathcal F$-BRWs (\cite[Proposition 4.5]{cf:BZ2} and \cite[Corollary 4.10(2)]{cf:BZ4}).
Morevoer, our result applies to BRWs which do not satisfy (U2) but can be projected onto a BRW which satisfy it.
As an example, consider a BRW on $\N$ such that $k_{n \, n+1}=1-1/2^{n+1}$, $k_{n\,n}:=1/2^{n+1}$ and $0$ otherwise.
This is an irreducible BRW which satisfies (U1) but not (U2) and can be projected
onto a one-point BRW with rate 1; thus Corollary~\ref{cor:main} applies.

\begin{teo}\label{th:mapping}
 Let $(X,K)$ be a continuous-time BRW such that (U2) holds. Suppose there exist
 $x_0\in X$, $Y\subseteq X$, $n_0\in\N$ such that
 \begin{enumerate}
  \item \label{n0} for all $x\in X$, $\min\{n\in \N\colon k_{xy}^{(n)}>0\text{ for some }y\in Y\}\le n_0$;
  \item \label{inject} for all $y\in Y$, there exists an injective map $\varphi_y:X\to X$ such that $\varphi_y(x_0)=y$ and 
  $k_{\varphi_y(x)\varphi_y(z)}\ge k_{xz}$ for all $x,z\in X$;
 \end{enumerate}
then (U1) holds. 
\end{teo}

\begin{proof}
Let $\delta=\inf\{k_{xy}\colon x,y \text{ such that }k_{xy}>0\}>0$ and put
$ \gamma_{xy}=k_{xy}/\delta$ for all $x,y\in X$. Define $\widehat T^n_x=\sum_{w\in X}\gamma_{xw}^{(n)}$
and $\widehat K_w:=\liminf_{n\to\infty}\sqrt[n]{\sum_{w\in X}\gamma_{xw}^{(n)}}$.
Clearly $\gamma_{xy}\ge1$ for all $x,y\in X$, $T^n_x=\delta^n \widehat T^n_x$ for all $x\in X$ and $n\in\N$,
$K_w=\delta \widehat K_w$; moreover $n \mapsto \widehat T^n_x$ is nondecreasing for all $x \in X$. 

Given $x\in X$, let $y(x)\in X$ be a vertex such that $k_{xy(x)}^{(m_x)}>0$ and
$m_{x}:=\min\{n\in \N\colon k_{xy}^{(n)}>0\text{ for some }y\in Y\}$.
Note that, by the hypothesis \eqref{inject}, we can map $x_0$ to $y(x)$ and get  
\begin{equation}\label{eq:Tn}
 \widehat T^n_{x_0}\le \widehat T_{y(x)}^n\le \gamma_{xy(x)}^{(m_x)}\sum_{w\in X}\gamma_{y(x)w}^{(n)}\le \widehat 
 T_{x}^{n+m_x}\le \widehat T_{x}^{n+n_0},
\end{equation}
for all $n\in\N$ (the second inequality is due to $\gamma_{xy(x)}^{(m_x)}\ge1$ and the last one holds since $m_x \le n_0$).
By irreducibility and the definition of $\widehat K_w$, given $\eps>0$, there exists $n_1\in\N$ such that
\begin{equation}
 \left(\widehat T^{n_1}_{x_0}\right)^{\frac{1}{n_1+n_0}}\ge \widehat K_w-\frac{\eps}{\delta}.
\end{equation}
Applying \eqref{eq:Tn},
\[
\left(\widehat T^{n_1+n_0}_{x}\right)^{\frac{1}{n_1+n_0}}\ge
\left(\widehat T^{n_1}_{x_0}\right)^{\frac{1}{n_1+n_0}}\ge
 \widehat K_w-\frac{\eps}{\delta}.
\]
Now,
\[
\left( T^{n_1+n_0}_{x}\right)^{\frac{1}{n_1+n_0}}=
\delta\left(\widehat T^{n_1+n_0}_{x}\right)^{\frac{1}{n_1+n_0}}
\ge  K_w-\eps
\]
thus
(U1) is satisfied.
\end{proof}

\section{Examples}\label{sec:examples}

This section is mainly devoted to examples of BRWs
where $\lambda_w=1/K_w$. In Section~\ref{subsec:treelike} we define the family
of periodic tree-like BRWs, where Theorem~\ref{thm:main} applies.
Next, in Section~\ref{subsec:BPVE} we view continuous-time branching processes in
varying environment as BRWs and determine their critical parameter.
Finally, Section~\ref{subsec:otherex} provides two examples (Examples~\ref{exmp:treeandZ}
and \ref{exmp:star}) showing that (U1) is not
necessary for $\lambda_w=1/K_w$, even when $K_w=\lim_{n\to\infty}\sqrt[n]{T_x^n}$,
both in the case where there is a weak phase ($\lambda_w<\lambda_s$) and where
there is not ($\lambda_w=\lambda_s$). Example~\ref{exmp:backtotheorigin} shows that even 
in the irreducible
case, it can be that $\lim_{n\to\infty}\sqrt[n]{T_x^n}$ does not exist and yet $\lambda_w=1/K_w$.
Example~\ref{exmp:finally} is the first example where $\lambda_w > 1/K_w$.
 
\subsection{Periodic tree-like BRWs}
\label{subsec:treelike}

We describe the construction of a class of irreducible BRWs that we call
\textit{periodic tree-like BRWs}.
Let $(I,E_I)$ be an irreducible finite oriented graph (possibly with loops) 
and let $\{(B_i, K(i))\}_{i \in I}$ be a family of finite and irreducible BRWs.
It might happen that even if $i \neq j$ then $(B_i, K(i))$ and $(B_j, K(j))$ are isomorphic BRWs.
Denote by $\{\varphi_{ij}\}_{(i,j) \in E_I}$ a family
of one-to-one maps 
from the domains $\mathrm{D}(\varphi_{ij})=:B_{ij}^- \subseteq B_i$ onto the images 
 $\mathrm{Im}(\varphi_{ij})=:B_{ij}^+\subseteq B_j$.
The main step of the construction is \textit{attaching an isomorphic copy of $B_i$ to an isomorphic copy of $B_j$}, that is, identifying a point
$x \in B_{ij}^-$ with $\varphi(x) \in B_{ij}^+$ for all $x \in  B_{ij}^-$ (each copy of $B_i$ is equipped with 
the same family of rates $K(i)$). 

We start by constructing recursively a labeled tree $({\mathcal{T}},E_{\mathcal{T}})$ which is going to be the skeleton of the BRW.
Denote by $i_0 \in I$ the \textit{root} of $I$ and let ${\mathcal{T}}_0:=\{(0,i_0)\}$ and $E(0):=\emptyset$; the label of $(0,i_0)$ 
is the projection $\pi((0,i_0)):=i_0$. Suppose we defined ${\mathcal{T}}_0, \ldots, {\mathcal{T}}_n$ and the set of edges $E(n) \subseteq \bigcup_{i=1}^n {\mathcal{T}}_i \times \bigcup_{i=1}^n {\mathcal{T}}_i$ and
suppose we defined the label $\pi(x)$ for all $x \in \bigcup_{i=1}^n {\mathcal{T}}_i$;
then ${\mathcal{T}}_{n+1}:= \{(x,i)\colon x \in {\mathcal{T}}_n, (\pi(x), i) \in E_I \}$ and $E(n+1):=E(n) \cup \{(x,(x,i))\colon (x,i) \in {\mathcal{T}}_{n+1}\}$. Moreover $\pi((x,i)):=i$ for
all $(x,i) \in {\mathcal{T}}_{n+1}$. Finally ${\mathcal{T}}:=\bigcup_{n \in \N} {\mathcal{T}}_n$ and $E_{\mathcal{T}}:=\bigcup_{n \in \N} E(n)$.
Roughly speaking, to each point in ${\mathcal{T}}_n$ of label $j$, we attach the same number of edges exiting $j$ in the graph $(I,E_I)$ and we
label the new endpoints accordingly; these new endpoints belong, by definition, to ${\mathcal{T}}_{n+1}$.

Note that, by construction, for every $i \in I$ there is an infinite number of vertices in $\mathcal{T}$ with label $i$;
moreover, since $(I,E_I)$ is connected and finite, the minimal number of steps required to go from $i$ to $i_0$ is bounded from above
by some $n_1$
with respect to $i$. The same bound holds for the minimal number of steps required to
go from a point with label $i$ in ${\mathcal{T}}$ to the closest point with label $i_0$.

We can construct now the periodic tree-like BRW. Let $\{(B_x,K(x))\}_{x \in {\mathcal{T}}}$ be a family of BRWs such that $(B_x,K(x))$ is an 
isomorphic copy of $(B_{\pi(x)}, K(\pi(x)))$  (we suppose that $B_x \cap B_y=\emptyset$ for all $x \neq y$).
For every $(x,y) \in {\mathcal{T}}$, we attach $B_x$ to $B_y$ as described above. The resulting irreducible BRW is denoted by $(X,K)$; note that, for this BRW,
condition (U2) holds since it is constructed by means of a finite number of types of BRWs.

Let us choose $x_0$ in the \textit{root set} $B_{(0,i_0)}$ where $(0,i_0)$ is the root of ${\mathcal{T}}$. 
We denote by $Y$ the set of the copies of the vertex $x_0$ (collected inside the copies
of the set $B_{(0,i_0)}$ inside $X$, namely $\{B_x\}_{x \in \mathcal{T} \colon \pi(x)=i_0}$).
Given any graph, let $d(x,y)$ be the minimal number of steps required to go from $x$ to $y$. Since $\{B_i\}_{i \in I}$ is a
finite family of finite sets we have that $n_2:=\sup_{i \in I, x,y \in B_i} d(x,y)<+\infty$.
Suppose we start from a vertex $x \in X$ and we want to reach the set $Y$; by construction, $x$ belongs to a copy of $B_{j_0}$
for some $j_0 \in I$. Let $\{j_0, \ldots, j_k\equiv i_0\}$ be the shortest path from $j_0$ to $i_0$ in $(I,E_I)$; clearly, $k \le n_1$. In order to reach
$Y$ from $x$ it is enough to exit the copy of $B_{j_0}$, to cross a copy of $B_{j_1}, \ldots, B_{j_{k-1}}$ (in this order),
to enter a copy of $B_{j_k}\equiv B_{i_0}$ and then to reach the copy of $x_0$ inside the last copy of $B_{j_k}$. Each one of these
actions requires at most $n_2$ steps. This implies that the length of the shortest path from $x$ to $Y$ is at most
$(n_1+1)n_2=:n_0$. Then Theorem~\ref{th:mapping} applies, and since $\sup_{x\in X}\sum_{y\in Y}k_{xy}<+\infty$
we can apply Theorem~\ref{thm:main} and get $\lambda_w=1/K_w$.

Here is an explicit example of such a construction: define $I:=\{1,2,3\}$, $i_0:=3$ and consider the graph $(I,E_I)$ pictured in
Figure~\ref{fig:1}. The corresponding pieces are shown in Figure~\ref{fig:2}; Figure~\ref{fig:3} explains how to join the pieces. 
The construction of the labeled tree $({\mathcal{T}},E_{\mathcal{T}})$ can be found in Figure~\ref{fig:4} and the final graph associated to the
periodic tree-like BRW is shown in Figure~\ref{fig:5} where we denoted by $x_0$ the ``actual $x_0$'' contained in the root set and all its copies.
\begin{figure}%

\begin{minipage}{0.48\textwidth}%
\centering
\includegraphics[width=0.48\textwidth]{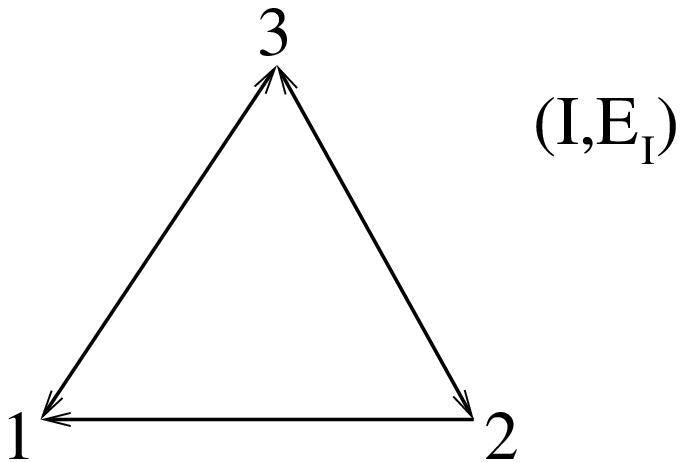}
\caption{The graph $(I,E_I)$.}%
\label{fig:1}%
\end{minipage}
\begin{minipage}{0.48\textwidth}%
\centering
\includegraphics[width=0.48\textwidth]{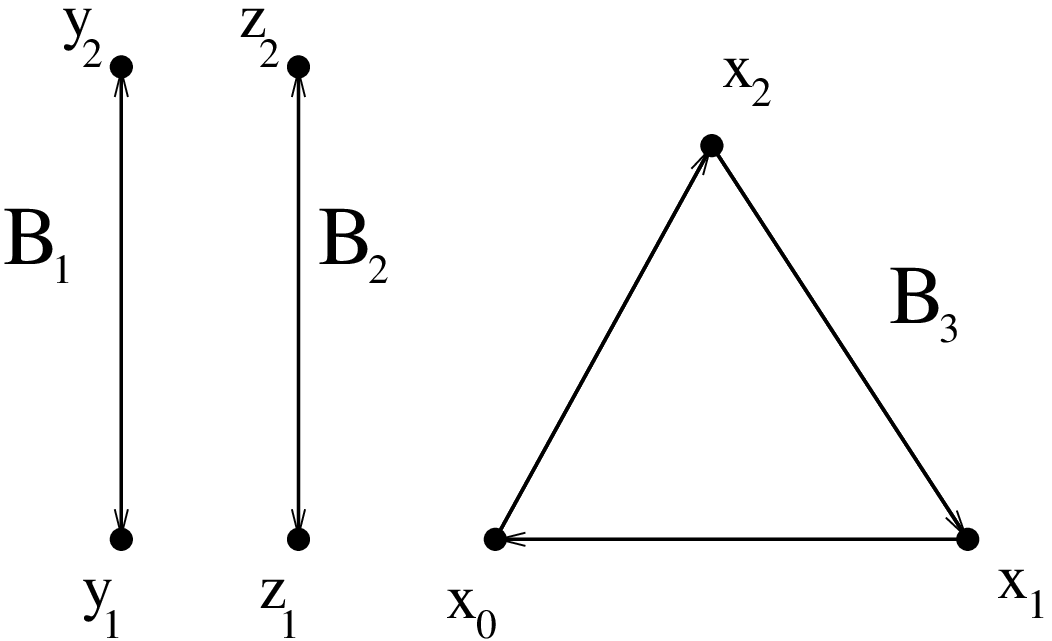}
\caption{The pieces $\{B_i\}_{i\in I}$.}%
\label{fig:2}%
\end{minipage}%
\end{figure}%

\begin{figure}%
\centering
\includegraphics[width=0.8\textwidth]{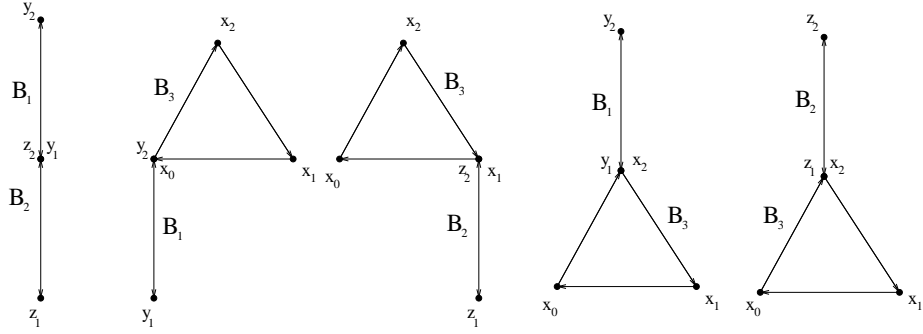}
\caption{How to join pieces.}%
\label{fig:3}%
\end{figure}%

\begin{figure}%
\begin{minipage}{0.29\textwidth}%
\centering
\includegraphics[width=0.8\textwidth]{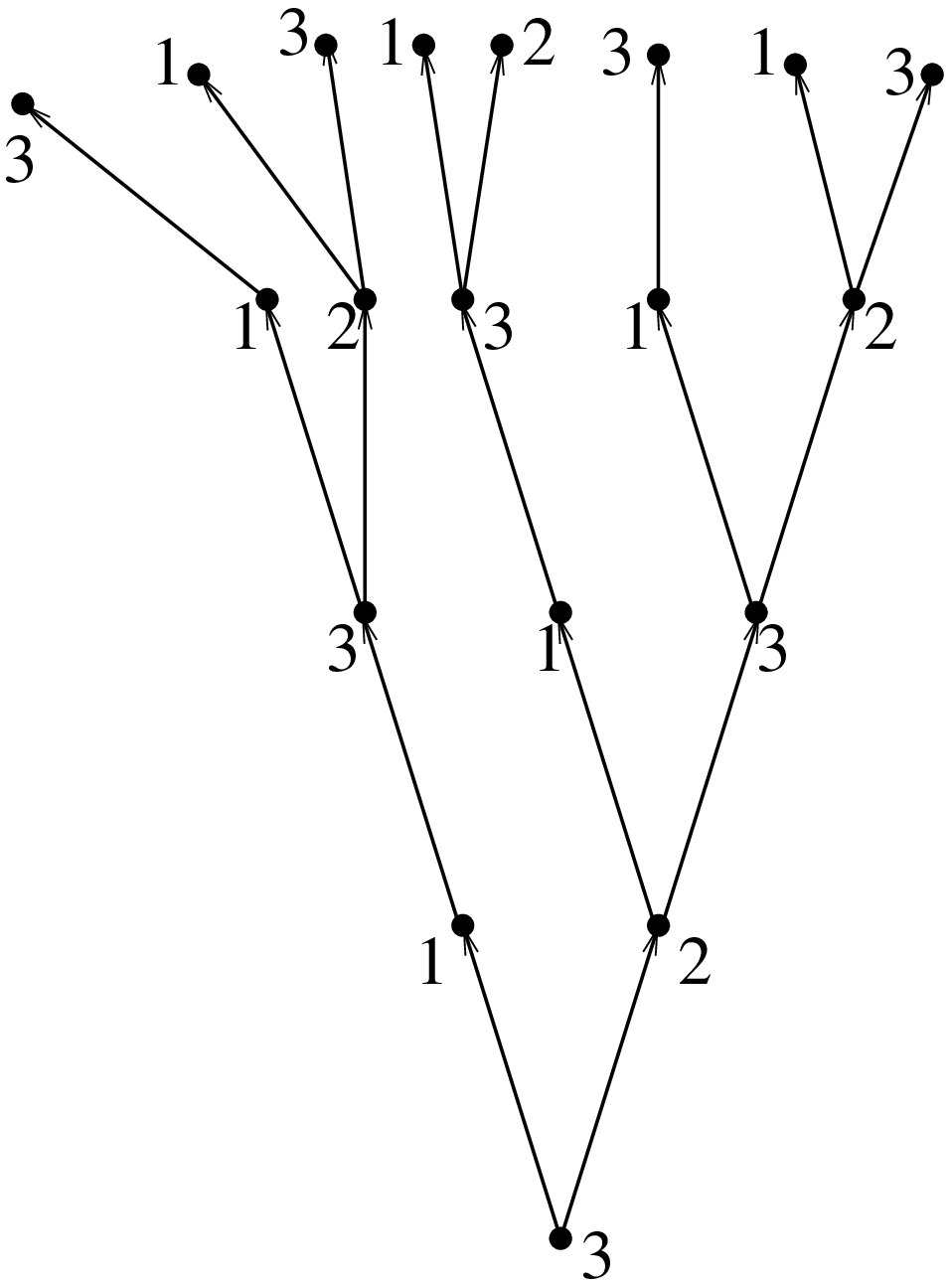}
\caption{The labeled tree $({\mathcal{T}},E_{\mathcal{T}})$.}%
\label{fig:4}%
\end{minipage}
\begin{minipage}{0.69\textwidth}%
\centering
\includegraphics[width=0.7\textwidth]{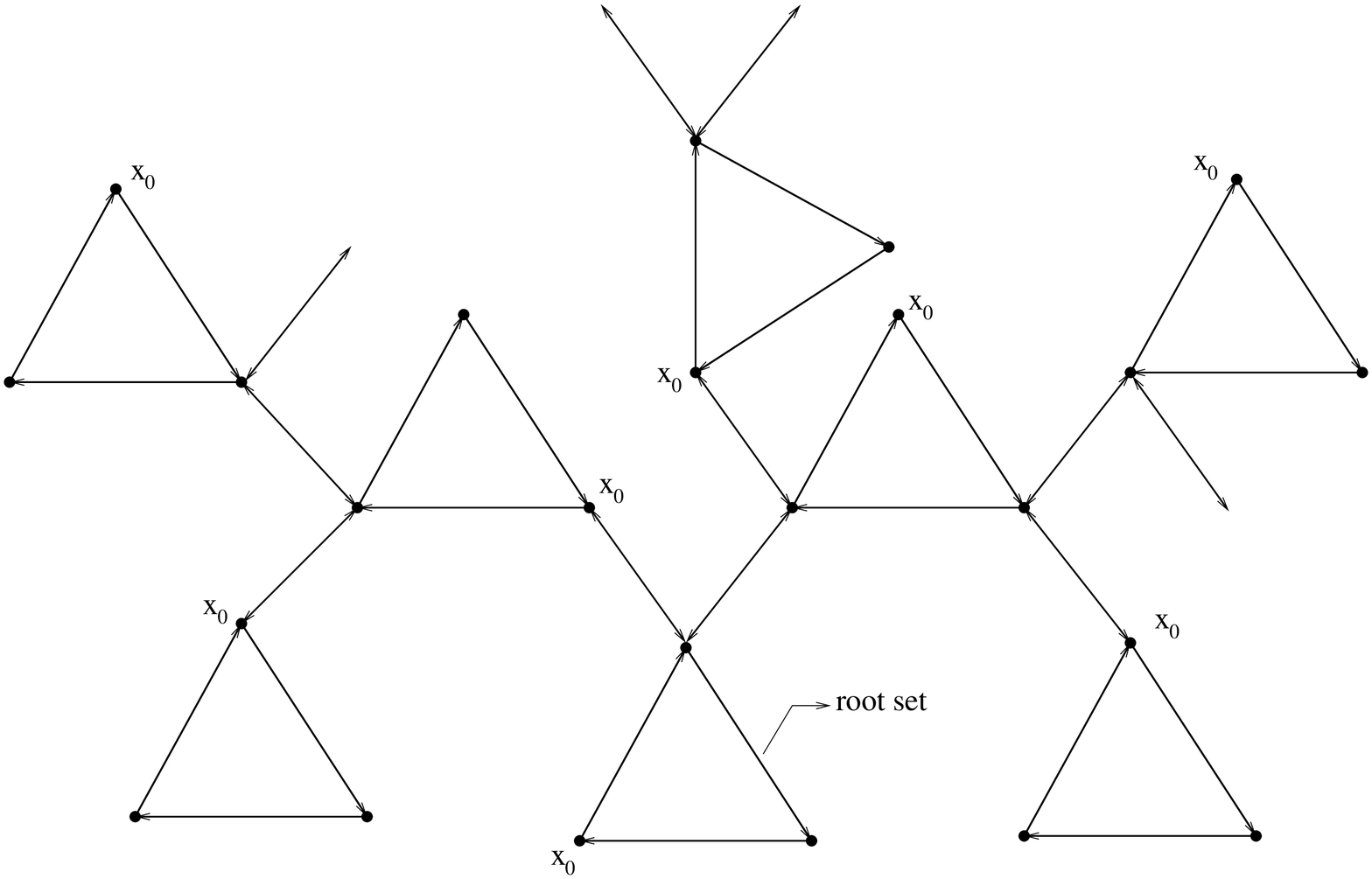}
\caption{The final periodic tree-like BRW.}%
\label{fig:5}%
\end{minipage}%
\end{figure}%

\subsection{Continuous-time branching process in varying environment}
\label{subsec:BPVE}

Consider a continuous-time
branching process where the breeding laws depend on the generation (while the death rate is always equal to 1);
this is called a \textit{branching process in varying environment} or \textit{BPVE}. 
To be precise, pick a sequence $\{k_n\}_{n \in \N}$ of strictly positive real numbers. The reproduction rate of a particle
of generation $n$ is $\lambda k_n$. By interpreting generations as space, the behavior of this process is
equivalent to the global behavior of a BRW on $\N$ where $k_{nm}=k_n$ if $n=m-1$ and $0$ otherwise.

We denote by $\lambda_w(n)$, as usual, the weak critical parameter of the associated BRW on $\N$ starting from $n$,
which is the only critical parameter for this process ($\lambda_s(n)=+\infty$ for all $n$ since 
there is clearly local extinction for every $n$).
Being the rates strictly positive, there is always a positive probability of reaching $n$ from $0$, hence $\lambda_w(0) \le \lambda_w(n)$.
On the
other hand in order to survive starting from $0$ the process has to pass by $n$ whence $\lambda_w(n) \le \lambda_w(0)$. Thus
$\lambda_w(n)=\lambda_w(0)$ for every $n \in \N$.
The generating function of the associated BRW is
\[
 G(\mathbf{q}|n)= \frac{1}{1+\lambda k_n (1-\mathbf{q}(n+1))}
\]
for all $\mathbf{q} \in [0,1]^\N$. Easy computations show that 
$K_w =\liminf_{n \to \infty} \sqrt[n]{\prod_{i=0}^{n-1} k_i}$. We are going to prove
in the next theorem that $\lambda_w=1/K_w$. We point out that in this case if $K_w=+\infty$ then $\lambda_w=0$ and there is survival
for every $\lambda>0$ (while it is always true that if $K_w=0$ then $\lambda_w=+\infty$).

\begin{teo}\label{th:continuousBPVA}
Let $\{X_t\}_{t \geq 0}$ be a continuous-time branching process in varying environment with reproduction rates $\{\lambda k_n\}_{n \in \mathbb{N}}$. 
Then $\lambda_w =1/{K_w}$.
\end{teo}

\begin{proof}
We study the survival of the branching process by analyzing its associated continuous-time BRW. 
Assume that $K_w \in (0,\infty]$. Since $\lambda_w \ge{1}/{K_w}$ for any BRW we just need to show that the 
reversed inequality holds. It follows from Theorem~\ref{th:survival}(2) that for this  purpose it suffices to find, for any 
$\lambda > 1/K_w$ (where $1/K_w=0$ if $K_w=+\infty$), a solution of $\lambda K{\bf{
v}} \geq {{\bf{v}}}/{(\bf{1}-{\bf{v}})}$ where 
${\bf{v}} \in [0,1]^\N$ 
such that ${\bf{v}} > {\bf{0}}$. Recall that $K{\bf{v}}(n) = k_n \mathbf{v}(n+1)$ for any $n$.
 Fix $\lambda > 1/K_w$ and choose 
$\rho \in (1/K_w,\lambda)$. 
Then define $\mathbf{v}(n)={t}/{(\rho^n \prod_{i=0}^{n-1} k_i)}$ where $t \le (1-\frac{\rho}{\lambda})/ M$ 
is fixed and $M$ is an 
upper bound of the sequence $\{{1}/{(\rho^n \prod_{i=0}^{n-1} k_i)}\}_{n \in \mathbb{N}}$ (whose existence is guaranteed by the fact that 
$\mathbf{v}(n)/t$ converges to $0$ as $n \rightarrow \infty$ by
our choice of $\rho$). Now
\begin{equation} 
\begin{split}
\lambda K{\bf{
v}}(n)&=\frac{\lambda k_n t}{\rho^{n+1} \prod_{i=0}^{n} k_i} 
= \frac{\lambda}{\rho} \cdot \frac{t}{\rho^n \prod_{i=0}^{n-1} k_i} 
\geq \frac{1}{1- tM}\cdot \frac{t}{\rho^n \prod_{i=0}^{n-1} k_i}\\
&\geq \frac{1}{1- t/(\rho^{n} \prod_{i=0}^{n-1} k_i)}\cdot \frac{t}{\rho^n \prod_{i=0}^{n-1} k_i}
=\frac{{\bf{v}}(n)}{1-{\bf{v}}(n)}.
\end{split}
\end{equation}
%
\end{proof}

\begin{rem}\label{rem:criticalbehaviour}
The identification of the critical parameter $\lambda_w=1/\liminf_{n \to \infty} \sqrt[n]{\prod_{i=0}^{n-1} k_i}$ does not tell us anything about the 
critical behavior when $\lambda=\lambda_w$. There is not just one possible scenario: in some cases there might be extinction
while in others there might be survival.

Indeed, suppose that $\lim_{n \to \infty} k_n=k \in (0,+\infty)$ and $k \ge k_n$ for every $n \ge n_0$ (for some $n_0 \in \N$). Then
by a simple coupling argument when $\lambda=\lambda_w=1/k$ the BRW starting from $n_0$ is stochastically bounded from above by
a BRW with rightward constant rate $1$ which is well-known to die out.

Conversely, consider $k_n:=(1+1/(n+1))^2$; then $\lambda_w=1$. Take $\lambda=1$
and define $\mathbf{v}(n):= 1/(n+2)$ for all $n \in \N$. We claim that $G(\mathbf{1}-\mathbf{v}) \le \mathbf{1}-\mathbf{v}$,
that is $\lambda K \mathbf{v} \ge \mathbf{v}/(\mathbf{1}-\mathbf{v})$;
indeed
\[
\begin{split}
\lambda K \mathbf{v}(n)-\frac{\mathbf{v}(n)}{1-\mathbf{v}(n)}& =k_n\mathbf{v}(n+1)-\frac{1/(n+2)}{1-1/(n+2)}
  =  \frac{(1+1/(n+1))^2}{n+3}-\frac{1}{n+1}\\
  &=\frac{1}{n+1} \Big ( \frac{(n+2)^2}{(n+1)(n+3)}-1 \Big )>0.
\end{split}
  \]
Hence by using $\mathbf{q}:=\mathbf{1}-\mathbf{v}$, according to Theorem~\ref{th:survival} there is global survival starting from any $n \in \N$.\end{rem}

\subsection{Other examples}
\label{subsec:otherex}

One of the main technical tools that we need in this section is 
the following theorem (see \cite[Theorem 2.2]{cf:BZ15}), which states that there cannot be a weak phase on slowly
growing BRWs. The proof, which we omit, makes use of equation~\eqref{eq:comparison} and the comment thereafter.

\begin{teo} \label{apgs}
Let $(X,K)$ be a continuous time, non-oriented BRW and let $x_0 \in X$. 
Suppose that there exists $\kappa \in (0,1]^X$ and $\{c_n\}_{n \in \mathbb{N}}$ such that, for all $n \in \mathbb{N}$ 
\begin{enumerate}
\item \label{apgs1} $\kappa(y)/\kappa(x_0) \leq c_n \ \forall y \in B(x_0,n)$
\item \label{apgs2} $\kappa(x)k_{x y} = \kappa(y) k_{y x} \ \forall x, y \in X$,
\end{enumerate}
where $B(x,n)$ is the ball of center $x$ and radius $n$ w.r.~to the natural distance of the graph 
$(X,E_K)$. If $\lim_{n\to\infty}c_n^{1/n} = 1$ and $\lim_{n\to\infty}|B(x_0,n)|^{1/n}= 1$, 
then $K_s(x_0,x_0) = K_w(x_0)$ 
and there is no pure global survival starting from $x_0$.
Moreover, in this case, $\liminf_{n\to\infty} \sqrt[n]{T_x^n}=\limsup_{n\to\infty} \sqrt[n]{T_x^n}$.
\end{teo}
In the proof of \cite[Theorem 2.2]{cf:BZ15} it was not explicitly mentioned that the
$\lim_{n\to\infty} \sqrt[n]{T_x^n}$ exists, but it 
follows easily 
by noting that $\liminf_{n\to\infty} \sqrt[2n]{k_{x_0x_0}^{(2n)}}=\limsup_{n\to\infty} \sqrt[2n]{k_{x_0x_0}^{(2n)}}$.

The following is an example where condition (U1) is not satisfied, nevertheless $\lambda_w=1/K_w<\lambda_s$
where, in this case, $K_w=\lim_{n\to\infty} \sqrt[n]{T_x^n}$.

\begin{exmp}\label{exmp:treeandZ}
Consider the 
irreducible continuous-time BRW on the graph obtained by identifying each vertex
of $\mathbb T_d$ ($d \geq 3$) with the vertex 0 of a copy of $\mathbb{Z}$ (each vertex is attached to a different
copy of $\Z$)  and let the rates matrix be the adjacency matrix of the graph. 
Denote this BRW by $(X,K_1)$.
We claim that $(X,K_1)$ can be projected into a BRW on $\mathbb{Z}$ with the following rates. 
Let $K$ be defined by $k_{0 0}:=d, k_{n\, n+1} := 1 =: k_{n+1\, n}$ and $0$ otherwise. Denote this new BRW by $(\Z,K)$.

Since $k_{n m} =1$ for $n \neq m$ whenever $|m-n|=1$, we conclude that $\kappa(n)k_{n m} = \kappa(m) k_{m n} \ \forall n, 
m \in \mathbb{N}$ if and only if $\kappa(n)=1 \ \forall n$. Then condition (\ref{apgs2}) in Theorem~\ref{apgs} is satisfied. 
Condition (\ref{apgs1}) in Theorem~\ref{apgs} is satisfied for any choice of $n_0$ by taking $c_n = 1$ for any $n \in \mathbb{N}$. 
Since $c_n^{1/n}=1$ 
and $|B(n_0,n)|^{1/n} \rightarrow 1$ as $n \rightarrow +\infty$ for any choice of $n_0$, we 
conclude that 
$K_s(\Z,K) =K_w(\Z,K) \geq 1/d$, where $K_w(\Z,K)=\lim_{n\to\infty} \sqrt[n]{\sum_{y \in X} k_{n_0 y}^{(n)}}$ 
(the existence of the limit is guaranteed by Theorem~\ref{apgs}).
Since projecting a BRW does not modified the value of the critical weak parameter neither the value of $K_w$, 
we conclude that $\lambda_w(X,K_1)=\lambda_w(\Z,K)=1/K_w(\Z,K)=
1/K_w(X,K_1) \leq 1/d$.  

Clearly, by going along a copy of $\Z$ in $X$ at arbitrarily long distance from the junction with $\mathbb{T}_d$, we have
that (U1) is not satisfied; more precisely, if $x(n)$ is a point in a copy of $\Z$ in $X$ at distance $n$ from the junction with 
$\mathbb{T}_d$, then $\sqrt[n]{\sum_{y \in X} k^{(n)}(x(n),y)}=2<d-\varepsilon \le K_w(\Z,K)-\varepsilon$ for all $\varepsilon \in (0,1)$.

Finally, we recall that the critical strong parameter can change 
in a projection; indeed it is not difficult to see 
that $\lambda_s(X,K_1)=1/(2\sqrt{d})>1/d \ge \lambda_w(X,K_1)$ for all $d \ge 3$. 
This can be proven by using the characterization
$\lambda_s(X,K_1)=\max\{\lambda \colon \Phi(x,x|\lambda) \le 1\}$ 
where, by standard generating function computations 
(see equation~\eqref{eq:HTheta} or the proof of \cite[Lemma 1.24]{cf:Woess}),
\[
 \Phi(x,x|\lambda)=1-\frac{d-2}{d-1}\sqrt{1-4\lambda^2}-\frac{d}{2(d-1)}\sqrt{1-4\lambda^2d},
\]
$x$ being a vertex in the tree $\mathbb{T}_d$.
\end{exmp}

In the following example, condition (U1) is not satisfied, nevertheless $\lambda_w=1/K_w=\lambda_s$; as before
$K_w=\lim_{n\to\infty} \sqrt[n]{T_x^n}$.

\begin{exmp}\label{exmp:star}
 Consider the 
 BRW obtained by attaching $d$ copies of the graph $\N$ to a common origin $0$ and by defining
 the rates according to the adjacency matrix as in the previous example. As before, this BRW can be projected onto $\N$ but 
 we discuss this example without any projection.
 As before, the key for computing $\lambda_w$ is to observe that Theorem~\ref{apgs} applies by taking $\kappa(x)=c_n = 1$ for
 all $n \in \N$ and every vertex $x$. This implies $\lambda_w=1/K_w=1/K_s=\lambda_s$ and $K_w=\lim_{n\to\infty} \sqrt[n]{\sum_{y \in X} k_{n_0 y}^{(n)}}$.
 What we have to do now is to compute $K_s$; we can do that by using the same technique as before. In this case
 \[
  \Phi(0,0|\lambda)=d \frac{1-\sqrt{1-4\lambda^2}}{2}
 \]
which implies $1/K_s=\max\{\lambda \colon \Phi(0,0|\lambda) \le 1\}=\sqrt{1/d-1/d^2}=\sqrt{d-1}/d$. Thus $K_w=k_s=d/\sqrt{d-1}$; thus,
as before,
if $x(n)$ is a point in a copy of $\Z$ at distance $n$ from the origin 
then $\sqrt[n]{\sum_{y \in X_1} k^{(n)}(x(n),y)}=2<K_w-\varepsilon$ for all $\varepsilon \in (0,K_w-2)$. Whence
(U1) does not hold.
\end{exmp}
The following is an example of an irreducible BRW on $\mathbb{N}$,
where $\lambda_w = 1 / K_w$ and $\lim_{n \to \infty} \sqrt[n]{T_x^n}$ does not exist. The idea is to pick outgoing rates which are
either 1 or 2 (alternating long stretches of 1s and 2s in order to keep the sequence oscillating).
Then we add rates from each $n$ to 0, so that the BRW is irreducible. If these rates are small enough, their presence
will neither affect $\lambda_w$ nor $K_w$.

\begin{exmp}\label{exmp:backtotheorigin}
Consider the BRW on $X=\mathbb{N}$ with the following rates. Let $K$ be defined by $k_{n n+1} := k_n \ge \delta>0, k_{n 0}:=\eps_n$ and 
$0$ otherwise. 
%
Observe that $\{T_0^n/\delta^n\}_{n \in \N}$ is nondecreasing; indeed, since $k_{x}/\delta \ge 1$,
\[
 \frac{T_0^n}{\delta^n} = \sum_{x \in \N} \frac{k_{0x}^{(n)}}{\delta^n} \le \sum_{x \in \N} \frac{k_{0x}^{(n)}}{\delta^n}\cdot 
 \frac{k_{x}}{\delta} \le \sum_{y \in \N} \frac{k_{0y}^{(n+1)}}{\delta^{n+1}}\le \frac{T_0^{n+1}}{\delta^{n+1}} .
\]
Choose the
sequence $\{\eps_n\}_{n \in \mathbb{N}}$ in such a way that $\beta := \sum_{i=0}^{+\infty} \eps_i (\prod_{j=0}^{i-1} k_j)/\delta^{i+1} < 1$. 
Since $T_0^0:= 1$ and
\begin{equation*}
\frac{T_0^n}{\delta^n} =  \prod_{j=0}^{n-1} \frac{k_j}{\delta} + 
 \sum_{i=0}^{n-2} \frac{\eps_i}{\delta} \Big (\prod_{j=0}^{i-1} \frac{k_j}{\delta} \Big ) \ \frac{T_0^{n-i-1}}{\delta^{n-i-1}} 
\le
 \prod_{j=0}^{n-1} \frac{k_j}{\delta} + 
 \sum_{i=0}^{n-2} \frac{\eps_i}{\delta} \Big (\prod_{j=0}^{i-1} \frac{k_j}{\delta} \Big ) \ \frac{T_0^{n}}{\delta^{n}}
\end{equation*}
we get
\begin{equation}
\prod_{j=0}^{n-1} k_j \leq T_0^n \leq \frac{\prod_{j=0}^{n-1} k_j}{1 - \beta} . 
\end{equation}
Therefore
\begin{equation}
K_w=\liminf_{n\to\infty} \left(T_0^n\right)^{1/n} = \liminf_{n\to\infty} \left(\prod_{j=0}^{n-1} k_j\right)^{1/n}.
\end{equation}

A straightforward computation shows that $\mathbf{v}(n):=1/ (\lambda^n \prod_{j=0}^{n-1}k_j)$ is a solution in $l^{\infty}(\mathbb{N})$ with $\mathbf{v} > 0$ of
$\lambda K \mathbf{v} \geq \mathbf{v}$ whenever 
$\lambda > 1/K_w$; equation~\eqref{eq:characterizationlambdaw} yields $\lambda_w \leq 1/K_w$.
Since $\lambda_w \geq 1/K_w$, we may conclude that $\lambda_w=1/K_w$. In \cite{cf:BZ2}, the following choice for the sequence 
$\{k_n\}_{n \in \mathbb{N}}$ is made. Define 
$a_n:=\lceil \log2/\log(1+1/n)\rceil, b_n:=\lceil \log2/(\log2 \log(2-1/n)) \rceil$ and $\{c_n\}_{n \in \mathbb{N}}$ recursively by
$c_1, c_{2r}=a_{2r}c_{2r-1}, c_{2r+1}=b_{2r+1}c_{2r}$ for any $r \geq 1$. Let $k_i := 1$ if $i \in (c_{2r-1},c_{2r}]$ (for some $r \in \mathbb{N}$)
and $k_i = 2$ if $i \in (c_{2r},c_{2r-1}]$ (for some $r \in \mathbb{N}$). It follows from this definition that
$\liminf_{n\to\infty} \left(\prod_{j=0}^{n-1} k_j\right)^{1/n}=1$ and that $\limsup_{n\to\infty} \left(\prod_{j=0}^{n-1} k_j\right)^{1/n}=2$. Therefore,
$\liminf_{n\to\infty} \left(T_0^n\right)^{1/n} = 1$ 
(note that in this last explicit example $\sum_{y \in X} k^{(n)}_{xy} \in [1,2^n]$ for all
$x \in X$ and $n\ge 1$, thus Theorem~\ref{thm:main} applies).
\end{exmp}

In the following example we have an irreducible BRW where $\lambda_w> 1/K_w$; moreover, we also provide an example of a reducible
BRW where $\lambda_w(x)> 1/K_w(x)$ and $\lambda_w(y)=1/K_w(y)$ for some $x,y \in X$ (nevertheless
 $\lambda_w(x)=\lambda_w(y)$ for all $x,y \in X$).
\begin{exmp}\label{exmp:finally}
To avoid confusion in the notation, in this example we denote the rates by $k_{x,y}$ instead of $k_{xy}$.
For simplicity we start with a reducible 
BRW on $\Z$: for all $n \in \N$ we take $k_{n,n+1} \in \{1,2\}$ as in Example~\ref{exmp:backtotheorigin} in such a way
 that $\liminf_{n \to \infty} \sqrt[n]{\prod_{i=0}^{n-1} k_{i,i+1}}=1<2= \limsup_{n \to \infty} \sqrt[n]{\prod_{i=0}^{n-1} k_{i,i+1}}$, while
 $k_{-n,-n-1}:=3-k_{n,n+1}$ for all $n \in \N$. Note that  $k_{{-}n,{-}n-1} k_{n,n+1}=2$ for all $n \in \N$;
 thus 
 \[
  \prod_{i=0}^{n-1} k_{-i,-i-1}=\frac{2^n}{\prod_{i=0}^{n-1} k_{i,i+1}}.
 \]
Hence, $\liminf_{n \to \infty} \sqrt[n]{\prod_{i=0}^{n-1} k_{-i,-i-1}}=1<2= \limsup_{n \to \infty} \sqrt[n]{\prod_{i=0}^{n-1} k_{-i,-i-1}}$. 
 Applying Theorem~\ref{th:continuousBPVA} to the process restricted to $\N$ and to the process restricted to $-\N:=\{-n \colon n \in \N\}$,
 we have that $\lambda_w(n)=1=1/K_w(n)$ for all $n \neq 0$.
 According to Remark~\ref{rem:lambdaw}, on the one hand $\lambda_w(0) \le \lambda_w(1)$ (since $0 \to 1$), 
 on the other hand (by taking $A:=\Z\setminus\{0\}$ in Remark~\ref{rem:lambdaw}) $\lambda_w(0) \ge \lambda_w(1)$; hence $\lambda_w(0)=1$.
 In order to compute $K_w(0)$ we note that 
\[
\sum_{i \in \Z} k^{(n)}_{0,i}=\prod_{i=0}^{n-1} k_{i,i+1}+\prod_{i=0}^{n-1} k_{-i,-i-1}=\prod_{i=0}^{n-1} k_{i,i+1} +\frac{2^n}{\prod_{i=0}^{n-1} k_{i,i+1}} 
\ge 2^{1+n/2},
\]
whence $K_w(0)=\liminf_{n \to \infty} \sqrt[n]{\sum_{i \in \Z} k^{(n)}_{0,i}} \ge \sqrt{2}$.
This implies that $\lambda_w(0) > 1/\sqrt{2} \ge 1/K_w(0)$; thus we have a reducible example where $\lambda_w(0)>1/K_w(0)$
and $\lambda_w(n)=1/K_w(n)$ for all $n\ge1$.

Let us modify this BRW to make it irreducible, as we did in Example~\ref{exmp:backtotheorigin}.
We add $k_{n,0}:=\eps_n$, such that 
$\beta_+ := \sum_{i=0}^{+\infty} \eps_i (\prod_{j=0}^{i-1} k_{j,j+1}) < 1/3$ and 
$\beta_- := \sum_{i=0}^{+\infty} \eps_{-i} (\prod_{j=0}^{i-1} k_{-j,-j-1}) < 1/3$.
Note that, as  in Example~\ref{exmp:backtotheorigin}, $\{T_0^n\}_{n \in \N}$ is nondecreasing (in this case $\delta=1$);
moreover $T_0^0=1$ and
\begin{equation*}
 \begin{split}
   {T_0^n} &=  \prod_{j=0}^{n-1} k_{j,j+1} + \prod_{j=0}^{n-1} k_{-j,-j-1}+
 \sum_{i=0}^{n-2} \Big [ \eps_i \Big (\prod_{j=0}^{i-1} k_{j,j+1} \Big )+\eps_{-i} \Big (\prod_{j=0}^{i-1} k_{-j,-j-1} \Big ) \Big ]\ T_0^{n-i-1} \\
  &\le 
  \prod_{j=0}^{n-1} k_{j,j+1} + \prod_{j=0}^{n-1} k_{-j,-j-1}+
  \sum_{i=0}^{n-2} \Big [ \eps_i \Big (\prod_{j=0}^{i-1} k_{j,j+1} \Big )+\eps_{-i} \Big (\prod_{j=0}^{i-1} k_{-j,-j-1} \Big ) \Big ]\ T_0^{n}.
 \end{split}
\end{equation*}
Hence
\[
 \prod_{j=0}^{n-1} k_{j,j+1} + \prod_{j=0}^{n-1} k_{-j,-j-1} \le
 {T_0^n}
 \le
 \frac{\prod_{j=0}^{n-1} k_{j,j+1} + \prod_{j=0}^{n-1} k_{-j,-j-1}}{1-\beta_+-\beta_-},
\]
and $K_w=\liminf_{n \to \infty} \Big (\prod_{j=0}^{n-1} k_{j,j+1} + \prod_{j=0}^{n-1} k_{-j,-j-1} \Big )^{1/n}\ge \sqrt2$.
We prove that $\lambda_w:=\lambda_w(\Z,K)=1$; indeed, suppose, by contradiction, that $\lambda_w<1$. 
For any fixed $\lambda \in (\lambda_w,1)$,
equation~\eqref{eq:characterizationlambdaw}
guarantees the existence $\mathbf{v} \in l^\infty(\Z)$ such that $\mathbf{v}(0)>0$ and 
$\lambda K \mathbf{v} \ge \mathbf{v}$.
In particular $\mathbf{v}_+$ (defined by $\mathbf{v}_+(n):=\mathbf{v}(n)$ for all $n \in \N$) satisfies 
$\lambda K^+ \mathbf{v}_+ \ge \mathbf{v}_+$ where $k^+_{0,1}=k_{0,1}+(3-k_{0,1}) \mathbf{v}(-1)/\mathbf{v}(1)$ and 
$k^+_{i,j}=k_{i,j}$ for $(i,j) \in \N^2\setminus\{(0,1)\}$;
this would imply $\lambda_w(\N,K^+) < 1$. Similarly,  
$\mathbf{v}_-$ (defined by $\mathbf{v}_-(n):=\mathbf{v}(-n)$ for all $n \in \N$) satisfies 
$\lambda K^- \mathbf{v}_- \ge \mathbf{v}_-$ where $k^-_{0,1}=k_{0,-1}+(3-k_{0,-1}) \mathbf{v}(1)/\mathbf{v}(-1)$ and 
$k^-_{i,j}=k_{-i,-j}$  for $(i,j) \in \N^2\setminus\{(0,1)\}$;
this would imply $\lambda_w(\N,K^-) < 1$. 
Note that $\min(k^+_{0,1},k^-_{0,1})\le 3$. Suppose, without loss of generality that $k^+_{0,1}\le 3$.
In this case, $\sum_{i=0}^{+\infty} \eps_i (\prod_{j=0}^{i-1} k^+_{j,j+1}) \le (3/k_{0,1}) \sum_{i=0}^{+\infty} \eps_i (\prod_{j=0}^{i-1} k_{j,j+1})<1$,
whence, by using the same above arguments, $K_w(\N,K^+)=\liminf_{n \to \infty} \sqrt[n]{\prod_{j=0}^{n-1} k^+_{j,j+1}}=1$; thus
$K_w(\N,K^+)=1>\lambda_w(\N,K^+)$ and this is a contradiction (as a consequence of Theorem~\ref{th:survival}).
\end{exmp}

\subsection{Random graphs}\label{subsec:randomgraphs}

In this section we give some ideas of what can happen in the case of a BRW on a random graph. There are examples of connected random graphs $\mathcal{G}$ where
$\lambda_w(\mathcal{G})$, $K_w(\mathcal{G})$ and $K_s(\mathcal{G})$ (hence $\lambda_s(\mathcal{G})$) are a.s.~constant. Take for instance 
a supercritical Bernoulli bond-percolation on $\Z^d$ and let $\mathcal{G}$ be the unique infinite cluster $C_\infty$. Consider now the BRWs on 
$\Z^d$ and $C_\infty$ where $K$ is the adjacency matrix. According to \cite[Theorem 1.1]{cf:BZ15}, $\lambda_s(C_\infty)=\lambda_s(\Z^d)$ and
 $K_s(C_\infty)=K_s(\Z^d)$ a.s.~(see also \cite[Section 7]{cf:BZ4}). This result is based on the fact that, in almost every realization of $C_\infty$, it is
possible to find arbitrarily large boxes of open edges and the parameter $K_s$ of the adjacency matrix of a box converges to $K_s(\Z^d)$ when the size
of the box goes to infinity (see \cite[Lemma 2.3]{cf:BZ15}, \cite[Section 3]{cf:BZ3} or \cite[Theorems 5.1 and 5.2]{cf:Z1}).
Analogously, according to \cite[Theorem 1.2]{cf:BZ15}, $\lambda_w(C_\infty)=\lambda_w(\Z^d)=1/K_w(C_\infty)=1/K_w(\Z^d)$ almost surely.

This might not be the case for a finite connected random graph $\mathcal{G}$; indeed, it is often possible to find
two (deterministic) graphs $G_1$ and $G_2$ such that $\pr(\mathcal{G}=G_i)>0$ (for $i=1,2$) and $K_s(G_1)=K_w(G_1)\neq K_w(G_2)=K_s(G_2)$
(which implies easily $\lambda_w(G_1)\neq \lambda_w(G_2)$ and $\lambda_s(G_1)\neq \lambda_s(G_2)$).
For an explicit example, consider a Bernoulli bond-percolation on the complete graph $\mathbb{K}_n$ with $n \ge 3$.
Given two different vertices $x,y \in \mathbb{K}_n$, there is a positive probability that the only (non-oriented) open edge in $\mathcal{G}$
is $(x,y)$; in this case $\lambda_s(\mathcal{G})=\lambda_w(\mathcal{G})=1/K_w(\mathcal{G})=1$.
On the other hand, given three different vertices $x,y,z \in \mathbb{K}_n$, there is a positive probability that the only (non-oriented) open edges in $\mathcal{G}$
are $(x,y)$ and $(y,z)$; in this case $\lambda_s(\mathcal{G})=\lambda_w(\mathcal{G})=1/K_w(\mathcal{G})=1/\sqrt{2}$.

\section*{Acknowledgements}

The first and third authors acknowledge support from INDAM-GNAMPA (Istituto Nazionale di Alta Matematica).
The second author thanks FAPESP (Grant 2015/20110-0) for financial support. This work was carried out during a stay of the second author at  
Dipartimento di Matematica, Politecnico di Milano. He is grateful for 
the hospitality and support. 
The authors are grateful to the anonymous referee for useful comments and suggestions.

\end{document}